\documentclass{article}
\usepackage{stylemin}

\newtheorem{theorem}{Theorem}

\newtheorem{lemma}{Lemma}
\theoremstyle{remark}
\newtheorem*{rem}{Remark}
\title{Mathematical study of the equatorial Ekman boundary layer}
\author{Jean Rax
\footnote{Sorbonne Université, Université Paris-Diderot SPC, CNRS,  Laboratoire Jacques-Louis Lions, LJLL, F-75005 Paris. rax@ljll.math.upmc.fr}}

\begin{document}

\maketitle

\begin{abstract}
In this paper we study the well-posedness of a simple model of boundary layer for rotating fluids between two concentric spheres near the equator. We show that this model can be seen as  a degenerate elliptic equation, for which we prove an existence result thanks to a Lax-Milgram type lemma. We also prove uniqueness under an additional integrability assumption and present a transparent boundary condition for such layers.
\end{abstract}

\section{Introduction}

In this article we will study the linear Ekman boundary layer near the equator for a rotating fluid between two concentric spheres. With $v$ the azimuthal flow velocity and $\psi$ the stream function, the equation we will consider is
\begin{equation}
\label{eq}
\begin{aligned}
& \d_z v + z \d_y v - \frac{1}{2} \d_y^4 \psi= s_\psi \\
& \d_z \psi + z \d_y \psi + \frac{1}{2} \d_y^2 v= s_v.
\end{aligned}
\end{equation}

We will mainly consider three domains and boundary conditions:

\begin{enumerate}[label=(\Roman*)]
\item \label{itm:casA} The domain is $y>0,z>0$ and the boundary condition at $z=0$ is $\psi_{|z=0}=0$.
\item \label{itm:casB} The domain is $y>0,H>z>0$, the boundary condition at $z=0$ is $\psi_{|z=0}=0$ and the one at $z=H$ is $\psi_{|z=0} = \Lambda v$, where $\Lambda: H^\frac{1}{2}_0\to H^{-\frac{1}{2}}$ is a non-positive operator.
\item \label{itm:casC} The domain is $y>0,z>H$ and the boundary condition at $z=H$ is $v_{|z=H}=v_H \in H^\frac{1}{2}_0$.
\end{enumerate}

Other cases can be obtained by altering the $z$ boundary conditions (for example $\psi_{z=0}=-\Lambda v_{|z=0}$). The boundary condition at $y=0$ will be
\begin{align*}
v_{|y=0}=V,\d_y \psi_{|y=0}=\Upsilon, \psi_{|y=0}=\Psi 
\end{align*}
and when not especially specified we will take $V=\Upsilon=0$, $\Psi=0$.

Since their description by Proudman~\cite{proudman1956almost} and their formal analysis by Stewartson~\cite{stewartson1957almost, stewartson1966almost} the behavior of highly rotating fluids have been widely mathematically studied. We refer to the book of Chemin, Desjardins, Gallagher and Grenier~\cite{chemin2006mathematical} for more details. Although the case of a horizontal surface (and the resulting $E^\frac{1}{2}$ boundary layer called Ekman layer) is now well understood,  especially since Grenier and Masmoudi~\cite{grenier1997ekman}, several other geometries have been considered. For a vertical surface (i.e the axis of rotation is perpendicular to the normal) the resulting boundary layer of size $E^\frac{1}{3}$ is well known and analysed (for a formal analysis see for example Van de Vooren~\cite{van1992stewartson}, and for a detailed analysis with anisotropic viscosity see Bresch, Desjardins and Gérard-Varet~\cite{bresch2004rotating}). In the spherical case the main difficulty is near the equator: as the latitude goes to $0$ the Ekman boundary layer degenerates and the classical analysis becomes invalid, leading to the need for an additional assumption of smallness near the equatorial area (as in the article by Rousset~\cite{rousset2007asymptotic}). It is to be noted that for small latitudes the $\beta$-plane model is used to take into account the variations of the angle between the axe of rotation and the normal of the surface as done by Dalibard and Saint-Raymond~\cite{dalibard2010mathematical}. 

In this paper we will focus on a linear and time independent model taking into account the spherical geometry (or any other similar geometry) in the vicinity of the equator. The resulting boundary layer (of size $E^\frac{1}{5} \times E^\frac{2}{5}$) was first derived by Stewartson~\cite{stewartson1957almost} and is a typical example of so called degenerate boundary layer~\cite{gerard2005formal, gerard2008remarks}. The derivation of the equation and its numerical analysis have been done notably by Marcotte, Dormy and Soward~\cite{marcotte2016equatorial}, and is briefly recalled in the appendix, but up to our knowledge no formal proof of the well posedness of the problem exists. 

For equation \eqref{eq} we will prove the existence in the natural energy space. We will also prove the uniqueness assuming additional integrability. For case~\ref{itm:casB} and for variants these additional assumptions are redundant and we have one and only one solution, however for case~\ref{itm:casA} and~\ref{itm:casC} the resulting space is smaller leading to an incomplete result.

A simplified statement of the existence and uniqueness result for $s_v=s_\psi=0$ is:

\begin{theorem}
For any $V,\Upsilon \in H^\frac{1}{2}_0(\mathbb{R}_+)$, $\Psi \in H^\frac{3}{2}_0(\mathbb{R}_+)$ there exists a weak solution of \eqref{eq} in cases~\ref{itm:casA},\ref{itm:casB},\ref{itm:casC}. This solution is such that
\begin{equation*}
\d_y v \in L^2(\Omega), \d_y^2 \psi \in L^2(\Omega).
\end{equation*}

Moreover if $v \in L^2, \psi \in L^2$ this solution is unique.
\end{theorem}

When uniqueness holds, a transparent boundary layer operator (similar to the one in~\cite{marcotte2016equatorial}) will be described. Such an operator is of great interest for numerical analysis or for connecting the boundary layer to the interior solution (or in this case to other boundary layers).

The main difficulty of the problem is that each approach to prove well-posedness stumbles on a different term. Let us observe the influence of each part of the equation:
\begin{itemize}
\item The $\d_z$ term is the obvious source of the degenerate character of the equation as a boundary layer equation as without it we recover a simple ODE with respect to $y$. More precisely at each $z$ we recover the classical Ekman layer. Even if its size diverges as $z \to 0$ one can make a formal expansion in powers of $\sqrt{z}$ and $y/\sqrt{z}$ which is the same as doing an expansion in $\frac{1}{\sqrt{\cos(\theta)}}$ for the Ekman problem.
\item The $z \d_y$ term associated to the fact that $y>0$ is also a major source of difficulties as it renders the spherical geometry. It prevents any simplification using symmetry arguments and as a counterpart of a simple domain it creates transport along characteristics $z-y^2/2=c$ which will create problems when trying to prove uniqueness. A possible approach would be to use well chosen weighted spaces that follow transport along those characteristics but we were unable to obtain satisfactory results with it. Without this transport term we recover the simple case of a vertical geometry.
\item The fact that the equation is a system and without a maximum principle prevents us from directly using modulated energy methods or entropy estimates that are usually helpful in such situations, for example in cross-diffusion problems.
\item Another main difference with standard cross-diffusion is the order of the operator in $y$: one term is a laplacian but the other is a bilaplacian. This asymmetry coupled with the boundary at $y=0$ is the main obstacle when trying to find better variables for the problem as the different regularity leads to mismatches in boundary conditions. The same problem arises when trying to use the decomposition between symmetric and skew-symmetric term for Carleman like estimate.
\end{itemize}

For these reasons, and the fact that the domain in $y$ is unbounded, our approach will be to consider the problem as a degenerate elliptic one. As a drawback this overlooks the structure of the skew-symmetric term containing the $\d_z$ and transport terms and leads to sub-optimal results in terms of regularity with respect to $z$.

In the first part we will deal with existence with a proof similar to the ones used by Fichera~\cite{fichera1959unified}. The main point is the use of a well-chosen space of test functions for a duality approach. This space must ensure both a coercivity condition for the adjoint via a positivity of boundary terms and the recovery of the boundary conditions which are weakly formulated. These two constraints dictate the set of admissible horizontal boundary conditions.

In the second part we will show the uniqueness by standard energy methods. We will also propose variations of the main problem allowing a uniqueness result in the same space as existence. For such variants we will define a transparent boundary operator similar to a Dirichlet to Neumann operator and of great importance for numerical simulation.

\section{Existence and properties of weak solutions}

In this section we prove the existence of weak solutions of \eqref{eq} using duality and energy methods for degenerate elliptic equations similar to the ones used by Fichera~\cite{fichera1959unified}. Thus, multiple boundary conditions can be weakly prescribed at $z=0$ and $z=H$, but the energy space is not regular enough to guarantee proper traces. 

In the rest of the article we will denote $\bm{u}=(v,\psi)$ and $\bm{s}=(s_\psi,s_v)$. The equation can then be formulated as $L \bm{u}=(T-\frac{1}{2} D) \bm{u}=\bm{s}$ where we defined the positive symmetric operator $D$ and the skew-symmetric operator $T$ as
\begin{align*}
T=\begin{pmatrix}
0&\d_z +z \d_y  \\
\d_z + z \d_y &0  
\end{pmatrix},
D= \begin{pmatrix}
\d_y^4 & 0       \\
0      & -\d_y^2 
\end{pmatrix}.
\end{align*}
The operator $D$ leads to the choice of the energy space $E_0$ and the operator $T$ prescribes both the allowed horizontal boundary conditions and the choice of the test function space. It is to be noted that the $z$ derivative and dependence is only involved in $T$, so it has no corresponding term in the energy space.

We will provide a detailed analysis for cases~\ref{itm:casA} and~\ref{itm:casB} and for homogeneous boundary conditions.

The other cases follow the exact same analysis, except for the choice of the space of test functions, which must be adapted to the horizontal boundary conditions. We will discuss nonhomogeneous boundary conditions in the next subsection.

\subsection{Statement of the result}
We define the Banach space $E_0$ by
\begin{align*}
\norm[E_0]{\bm{u}}^2  = \int_\Omega \left( \left|\d_y v \right|^2 + \left|\frac{v}{1+y} \right|^2\right) +\int_\Omega \left( \left|\d_y^2 \psi \right|^2 + \left|\frac{\psi}{1+y^2} \right|^2\right) 
\end{align*}
and to enforce homogeneous boundary conditions at $y=0$ we define
\begin{equation*}
\norm[E_{0,0}]{\bm{u}}^2  = \int_\Omega \left( \left|\d_y v \right|^2 + \left|\frac{v}{y} \right|^2\right) +\int_\Omega \left( \left|\d_y^2 \psi \right|^2 + \left|\frac{\psi}{y^2} \right|^2\right).
\end{equation*}
Lastly for the weak formulation we denote the graph norm
\begin{equation*}
\norm[E_1]{\bm{u}}    =\norm[E_{0,0}]{\bm{u}}+\norm[E_{0,0}']{T \bm{u}}  .     
\end{equation*}
Note that $E_0$ lacks regularity with respect to $z$ to have traces at $z=0$ (or $z=H$). Moreover, $\bm{u} \in E_1$ requires not only some weak (negative) regularity on $\d_z \bm{u}$ but also a better integrability than just $E_0$.

An important point is that these $z$ boundary conditions are derived from the space of test functions. Let us consider 
\begin{equation*}\mathcal{D}=\{\bm{\varpi}(y,z)=(w(y,z),\phi(y,z)); w \in \Cinf((0,+\infty)\times [0,+\infty)), \phi \in \Cinf((0,+\infty)\times (0,+\infty)) \}\end{equation*}
for case~\ref{itm:casA} and 
\begin{equation*}\mathcal{D}=\{\bm{\varpi}=(w,\phi); w \in \Cinf((0,+\infty)\times [0,H]), \phi \in \Cinf((0,+\infty)\times (0,H]) \text{ s.t } \phi_{|z=H}=-\Lambda^* w_{|z=H} \}\end{equation*}
for case~\ref{itm:casB}. Note that in fact we can replace $\mathcal{D}$ by its closure under the $E_1$ norm.

Given this set of definitions, the following existence theorem holds, where as for the rest of the paper $C$ denotes a numerical constant
\begin{theorem}
\label{existence}
Let $\bm{s}=(s_v,s_\psi) \in E_1'$.
\begin{enumerate}[label=(\roman*)]
\item \label{itm:thi} (existence of weak solutions)
There exists $\bm{u} \in E_{0,0}$  such that $\forall \bm{\varpi}=(w,\phi) \in \mathcal{D}$
\begin{align}
\label{eqweak}         
\int_\Omega  - v  \d_z\phi - z v  \d_y\phi -\frac{1}{2} \d_y^2 \psi \d_y^2 \phi +\int_\Omega - \psi  \d_z w - z  \psi  \d_y w -\frac{1}{2} \d_y v \d_y w     = \int_\Omega s_\psi \phi+\int_\Omega s_v w 
\end{align}
and
\begin{align*}
\norm[E_{0,0}]{\bm{u}} \leq C \norm[E_1']{\bm{s}} 
\end{align*}

\item \label{itm:thii} (boundary conditions) If $\bm{u} \in E_0 \cap H^2_{loc}$ then $\psi_{|z=0}=0$ and in case~\ref{itm:casB}, $\psi_{|z=H}=\Lambda v_{|z=H}$.

\item \label{itm:thiii} (interior regularity)
If $\d_y^2 \bm{s} \in E_0'$ we have a Caccioppoli type inequality: for all $y_0>0,z_1>0$ there exist $C_{y_0,z_1}>0$ such that
\begin{equation*}
\int_{(y_0,\infty)\times(0,z_1)} \left| \d_y^4 \psi \right|^2+ \left| \d_y^3 v \right|^2 \leq C_{y_0,z_1} \left(\norm[E_0]{\bm{u}}^2+\norm[E_0']{\d_y^2 \bm{u}} \right)
\end{equation*}
\end{enumerate}
\end{theorem}
\begin{rem} as we have
\begin{align*}
& \d_z v =-z \d_y v + \frac{1}{2}\d_y^4 \psi +s_\psi            \\
& \d_z \d_y \psi = -z \d_y^2 \psi -\frac{1}{2}\d_y^3 v+\d_y s_v 
\end{align*}
from the interior regularity with respect to $y$ we can obtain interior regularity with respect to $z$.
\end{rem}

Points~\ref{itm:thi},~\ref{itm:thii} and~\ref{itm:thiii} will be proved in subsection~\ref{subsection:duality},~\ref{subsection:BCcoherence} and~\ref{subsection:caccioppoli} respectively.
\subsection{Remarks on nonhomogeneous boundary conditions}			

The previous result only considers homogeneous boundary conditions. As usual we can recover nonhomogeneous boundary condition by lifting these boundary conditions. In this subsection we will briefly discuss this lifting.

Note that an important difference with Ekman boundary layers is that we are able to impose $3$ boundary conditions at $y=0$ whereas in classical Ekman boundary layers only $2$ boundary conditions can be imposed. 
This difference does not come from any particularity of our system  as the same number can be prescribed if we replace the transport $\d_z + z \d_y$ by $\lambda u + c \d_y$ with $\lambda,c \neq 0$. 
On the contrary, one can only prescribe 2 $2$ conditions for Ekman layers. this comes from a degeneracy of the Ekman system, which causes the Ekman pumping.

In order to consider nonhomogeneous boundary conditions we will need the following lemma:
\begin{lemma}
\begin{itemize}
\item Let $\Psi \in H^\frac{3}{2}_0(\mathbb{R}_+)$ i.e $\Psi \in  H^\frac{3}{2}$ and $\Psi(0)=0$. Moreover suppose that $z \Psi \in H^\frac{3}{2}$. Then there exist $\bm{r} \in E_0$ such that
\begin{align*}
&L \bm{r} \in E_1' \text{ and } \norm[E_1']{L \bm{r}} \leq C \norm[H^\frac{3}{2}_0]{(1+z)\Psi} \\
\bm{r}_{|z=0}=0,~& r_{v|y=0}=0,~ r_{\psi|y=0}=\Psi,~ \d_y r_{\psi|y=0}=0
\end{align*}
\item Let $\Upsilon \in H^\frac{1}{2}_0(\mathbb{R}_+)$ i.e such that $\Upsilon \in H^\frac{1}{2}$ and
\begin{equation}
\label{compatpsi}
\int_0^1 \frac{|\Upsilon|^2(z)}{z} dz < + \infty.
\end{equation}
Let suppose moreover $z \Upsilon \in H^\frac{1}{2}(\mathbb{R}^+)$. Then there exist $\bm{r}=(r_v,r_\psi) \in E_0$ verifying
\begin{align*}
&L \bm{r} \in E_1' \text{ and } \norm[E_1']{L \bm{r}} \leq C \norm[H^\frac{1}{2}_0]{(1+z)\Upsilon} \\
\bm{r}_{|z=0}=0,~& r_{v|y=0}=0,~ r_{\psi|y=0}=0,~ \d_y r_{\psi|y=0}=\Upsilon
\end{align*}
\item Let $V,v_0 \in H^\frac{1}{2}(\mathbb{R}_+)$ such that
\begin{equation}
\label{compatv}
\int_0^1 \frac{|V(\zeta)-v_0(\zeta)|^2}{\zeta} d\zeta <+\infty
\end{equation}
and $zV \in H^\frac{1}{2}$. Then there exist $\bm{r}=(r_v,r_\psi) \in E_0$ verifying
\begin{align*}
&L \bm{r} \in E_1' \\
r_{v|z=0}=v_0,~r_{\psi|z=0}=0,~& r_{v|y=0}=V,~ r_{\psi|y=0}=0,~ \d_y r_{\psi|y=0}=0
\end{align*}
and
\begin{equation*}
\norm[E_1']{L \bm{r}} \leq C \left( \norm[H^\frac{1}{2}]{(1+z)V}+\norm[H^\frac{1}{2}]{v_0}+ \sqrt{\int_0^1 \frac{|V(\zeta)-v_0(\zeta)|^2}{\zeta} d\zeta}\right).
\end{equation*}
\end{itemize}
\end{lemma}

The proof is exactly the same as the one of theorem 1.5.2.4 in Grisvard's book~\cite{grisvard2011elliptic}. Once the compatibility conditions \eqref{compatpsi},\eqref{compatv} are verified, one can find $r_v \in H^1(\Omega)$ and $r_\psi \in H^2(\Omega)$ verifying the boundary conditions.
The only difference is that we first need to lift $(1+z) (V,\Upsilon,\Psi)$ and then divide the lifting by $(1+z)$ to obtain the correct integrability of $z \d_y r$.

Note that there is no physical sense of non zero $\psi_{y|=0}$ in our problem. In fact this corresponds to the non penetration condition, and a non zero $\Psi$ will create a pumping similar to the Ekman pumping. However we included this case for the sake of mathematical completeness.

Moreover these hypotheses are far from optimal, in fact we recover more regularity with respect to $z$ than needed for $L \bm{r}$.

Once this lemma is established, by linearity, considering the equation for $\bm{u}-\bm{r}$ with source term $L \bm{r}$ we can solve the equation with source terms satisfying the hypothesis of the lemma.

In the rest of the paper we will thus consider only homogeneous boundary conditions at $y=0$.

For a nonhomogeneous horizontal boundary condition in case~\ref{itm:casC} condition \eqref{compatv} then becomes
\begin{equation*}
v_{|z=H}=v_H \in H^\frac{1}{2}_0.
\end{equation*}
This condition will be used in the formulation of the transparent boundary condition.

\subsection{Duality principle (proof of~\ref{itm:thi})}
\label{subsection:duality}
To prove the first part of theorem~\ref{existence} we will consider the equation as an elliptic equation, albeit a degenerate one. It will allow us to use classical functional analysis and to carefully encode the boundary conditions in the functional spaces.

Equation \eqref{eqweak} can be seen formally as the following problem: find $\bm{u} \in E_{0,0}$ such that $\forall \bm{\varpi}=(w,\phi) \in  \mathcal{D} \subset E_1$,
\begin{align*}
\dual{ L \bm{u}}{\bm{\varpi}}_{E_1',E_1} & = \int_\Omega  - v  \d_z\phi - z v  \d_y\phi -\frac{1}{2} \d_y^2 \psi \d_y^2 \phi +\int_\Omega - \psi  \d_z w - z  \psi  \d_y w -\frac{1}{2} \d_y v \d_y w \\
& = +\int_\Omega s_\psi \phi+\int_\Omega s_v w=\dual{\bm{s}}{\bm{\varpi}}_{E_1',E_1}                                                                                   
\end{align*}

where $L:E_{0,0} \to E_1'$ is a continuous linear operator as $\|u\|_{E_0} \leq \|u\|_{E_{0,0}}$ and
\begin{equation*}
\left\vert \dual{L\bm{u}}{\bm{\varpi}}_{E_1,E_1'} \right\vert=\left\vert \int_\Omega \bm{u} \cdot (T\bm{\varpi})+\frac{1}{2}\int_\Omega (\d_y^2 v \d_y^2w+\d_y \psi \d_y \phi) \right\vert \leq (\|\bm{u}\|_{E_{0,0}}\|T \bm{\varpi}\|_{E_{0,0}'}+\|\bm{u}\|_{E_0} \|\bm{\varpi}\|_{E_0}).
\end{equation*}

Through usual functional analysis methods (typically Lions-Lax-Milgram theorem, see lemma~\ref{analfunc} for details) we have at least one solution of $L \bm{u} = \bm{s}$ as long there exist a coercivity inequality for the adjoint operator $L^*: E_1 \to E_{0,0}'$, i.e a constant $C$ such that

\begin{align*}
\forall \bm{\varpi} \in \mathcal{D} \subset E_1:         \\
\norm[E_{0,0}']{L^* \bm{\varpi}} \geq C \norm[E_1]{\bm{\varpi}} . 
\end{align*}
We have for $\bm{\varpi} \in \mathcal{D}$
\begin{align*}
\dual{ L^* \bm{\varpi} }{\bm{\varpi}}_{E_{0,0}',E_{0,0}} & = \int_\Omega - \d_z w  \phi- z \d_y w  \phi -\frac{1}{2} \d_y^2 \phi \d_y^2 \phi +\int_\Omega -\d_z \phi  w - z \d_y \phi  w -\frac{1}{2} \d_y w \d_y w \\
& =-\frac{1}{2} \int_\Omega |\d_y^2 \phi |^2 + |\d_y w|^2-  \int_\Omega \d_z (w \phi)+z \d_y (w \phi)     .                                                 
\end{align*}
Using the fact that $w$ and $\phi$ are in $\Cinf((0,+\infty)\times [0,+\infty))$ Hardy's inequality (see for example~\cite{masmoudi2011hardy}) reads as
\begin{align*}
\frac{1}{2} \int_\Omega |\d_y^2 \phi |^2 + |\d_y w|^2 \geq C \norm[E_{0,0}]{\bm{\varpi}}^2.
\end{align*}
The first term arising from the skew-symmetric part $T$ is
\begin{align*}
\int_\Omega z \d_y (w \phi) = 0.
\end{align*}
The last term arising from the skew-symmetric part $T$ is 
\[-\int_\Omega \d_z ( w \phi).\]

For case~\ref{itm:casA} this term is $0$ thanks to the boundary condition i.e the fact that $\bm{\varpi} \in \mathcal{D}$.

For case~\ref{itm:casB} we have
\begin{align*}
-\int_\Omega \d_z (w \phi)=\int_{y}  w_{|z=H} \Lambda^*w_{|z=H} dy \leq 0 . 
\end{align*}

So for all cases
\begin{align*}
\dual{L^* \bm{\varpi}}{\bm{\varpi}}_{E_{0,0}',E_{0,0}} \leq -C\norm[E_{0,0}]{\bm{\varpi}}^2 
\end{align*}
leading to the inequality
\begin{align*}
\norm[E_{0,0}]{\bm{\varpi}} \leq C \norm[E_{0,0}']{L^* \bm{\varpi}}.
\end{align*}

As $\bm{\varpi} \in \mathcal{D}$ we have $T \bm{u} =-L^* \bm{u} -\frac{1}{2} D \bm{u} $ and as $\norm[E_0']{D \bm{\varpi}} \leq c \norm[E_{0,0}]{\bm{\varpi}}$ 

\begin{align*}
\norm[E_1]{\bm{\varpi}} & =\norm[E_{0,0}]{\bm{\varpi}}+\norm[E_{0,0}']{T \bm{\varpi}} \leq \norm[E_{0,0}]{\bm{\varpi}}+\norm[E_{0,0}']{L^* \bm{\varpi}}+c \norm[E_{0,0}]{\bm{\varpi}} \leq C \norm[E_{0,0}']{L^* \bm{\varpi}} . 
\end{align*}

We recognize the coercivity inequality needed to prove the point $(i)$ of the theorem.

It can be checked that all the other cases can be analyzed along the very same lines, the main point and only part where $z$ boundary conditions appear being the sign of $-\int_\Omega \d_z ( w \phi)$.
\subsection{Boundary conditions (proof of~\ref{itm:thii})}
\label{subsection:BCcoherence}

As functions in the energy space $E_0$ do not display sufficient regularity to have traces at $z=0$, we used the duality formulation to prescribe such boundary conditions in a weak sense. For example $\psi_{|z=0}=0$ means that for all $w \in C^\infty_c((0,+\infty)\times[0,\infty])$ we have
\begin{equation*}
\int_\Omega \psi \d_z w = 0
\end{equation*}

It remains to demonstrate that for a sufficiently regular solution this formulation is equivalent to the aforementioned boundary conditions.

To do so, let us consider $\bm{u} \in E_0 \cap H^2_{loc}$ a solution of \eqref{eqweak} (note that all weak solutions for a smooth source term have interior regularity by the point $(iii)$). Then all considered traces are well defined.

Let $h$ be a regular function such that $h(0)=1, \operatorname{supp} h \subset [0,1)$ and $g \in C_c^\infty((0,+\infty))$. With $\bm{\varpi}_\eta=(w_\eta,\phi_\eta)=\left(g(y) h \left( \frac{z}{\eta} \right),0\right)=(g(y) h_\eta(z),0) \in \mathcal{D}$ used  as a test function we get
\begin{align*}
-\int_\Omega \psi \d_z w_\eta+\int_\Omega \left(-z \d_y w_\eta \psi   +\frac{1}{2} \d_y^2 w_\eta v\right)  = \int_\Omega s_v w_\eta . 
\end{align*}
As $\norm[E_0]{\bm{\varpi}_\eta} \to 0$ when $\eta \to 0$,
\begin{align*}
\lim_{\eta \rightarrow 0} \int_\Omega \psi \d_z w_\eta = 0 . 
\end{align*}
But $\d_z h_\eta$ is approximating the identity, so
\begin{align*}
\int_\Omega \d_z \psi w_\eta \to  \int_0^\infty g(y) \psi(y,0) dy 
\end{align*}
thus for all smooth $g$
\begin{align*}
\int_{0}^\infty g(y) \psi(y,0) dy = 0 
\end{align*}
i.e $\psi(y,0)=0$.

Similarly for case~\ref{itm:casB} with
\begin{align*}
& w =  g(y) h_\eta(z)              \\
& \phi = -\Lambda^* g(y) h_\eta(z) 
\end{align*}
as $\eta$ goes to $0$ we obtain
\begin{align*}
\int_\Omega \left(\Lambda^* g(y) v  -\psi  g(y) \right) \d_z h_\eta(z) =0 
\end{align*}
which leads to
\begin{align*}
\int_0^\infty \left(\Lambda v_{y,H}  -\psi_{y,H} \right)(y) g(y) dy=0 
\end{align*}

for all  $g$. This is the expected result.

As for the previous point the other cases (notably $v_{|z=H}=0$) can be described along the same lines.

\subsection{Caccioppoli type inequality and interior regularity}
\label{subsection:caccioppoli}					

In order to obtain interior regularity we use the elliptic character with respect to $y$ (associated with the $D$ part of the linear operator) to obtain Caccioppoli type inequalities with respect to $y$.

Let $\theta$ a smooth function on $\mathbb{R}$ such that
\begin{equation*}
\theta(\zeta)\begin{cases}=0 \text{ if } \zeta \in\left(-\infty, 0 \right]\\
\in [0,1] \text{ if } \zeta \in \left( 0,1\right)\\
=1 \text{ if } \zeta \in\left[1, +\infty \right)
\end{cases}
\end{equation*}

Let $z-1>0$, $L>y_0>0$ and define
\begin{equation*}
\chi(y,z)=\theta \left(\frac{2z_1-z}{z_1} \right) \theta\left(\frac{2y-y_0}{y_0} \right) \theta	\left(\frac{2L-y}{L} \right)
\end{equation*}
then $\chi$ is a smooth cut-off function such that $\chi =0$ outside $(y_0/2,2L) \times [0,2 z_1)$ and $\chi=1$ inside $[y_0,L] \times [0,z_1]$.

Let $\rho_\eps$  an approximation of the identity with support inside $\mathbb R^-$ and $\bm{u}_\eps=\rho_\eps *_y \bm{u}$. We have $\bm{u}_\eps$ smooth with respect to $y$ and solution of equation \eqref{eq} with a source term $\bm{s}_\eps=\rho_\eps *_y \bm{s}$. Using the equation we can deduce that $\d_z \bm{u}_\eps$ is also smooth with respect to $y$ so $\d_y^2 \bm{u}_\eps \chi^4$ and its derivatives with respect to $y$ are in $E_1'$.

With $\d_y^2( \d_y^2 \bm{u}_\eps \chi^4)$ as a test function, integrating by parts we obtain after cancellation of most skew-symmetric terms

\begin{align*}
-\int \d_y^2 v_\eps  \d_y^2 \psi_\eps \d_z \chi^4 &- \int z \d_y^2 v_\eps \d_y^2 \psi_\eps \d_y \chi^4-\int \d_y^6 \psi_\eps \d_y^2 \psi_\eps \chi^4+ \int \d_y^4 v_\eps \d_y^2 v_\eps \chi^4 =\\
& 2 \int \d_y^2 {S_\eps}_\psi \d_y^2 \psi_\eps \chi^4 +\d_y^2 {S_\eps}_v \d_y^2 v_\eps \chi^4. 
\end{align*}
Integrating by parts again, we get
\begin{align*}
\int \d_y^4 v_\eps \d_y^2 v_\eps \chi^4       & = - \int |\d_y^3 v_\eps|^2 \chi^4 +\frac{1}{2} \int \d_y^2 v_\eps \d_y^2 v_\eps \d_y^2 \chi^4\\
&= - \int |\d_y^3 v_\eps|^2 \chi^4 -\frac{1}{2} \int \d_y^3 v_\eps \d_y v_\eps \d_y^2 \chi^4+\frac{1}{4} \int |\d_y v_\eps |^2 \d_y^4 \chi^4 \\
\int \d_y^6 \psi_\eps \d_y^2 \psi_\eps \chi^4 & = \int |\d_y^4 \psi_\eps|^2 \chi^4 +\int \d_y^4 \psi_\eps \d_y^2 \psi_\eps \d_y^2 \chi^4 \psi_\eps-\int \d_y^3 \psi_\eps \d_y^3 \psi_\eps \d_y^2 \chi^4                                                                                  \\
& = \int |\d_y^4 \psi_\eps|^2 \chi^4 +2 \int \d_y^4 \psi_\eps \d_y^2 \psi_\eps \d_y^2 \chi^4 - \frac{1}{2}\int |\d_y^2 \psi_\eps|^2 \d_y^4 \chi^4.                                                                                          
\end{align*}
Moreover, defining $\eta_1^{-1} =16 \|\theta'\|^2_\infty \left(\frac{1}{z_1}+z_1 \left(\frac{2}{y_0}+\frac{1}{L}\right) \right) \geq 16 \sup |(\d_z+z \d_y) \chi|^2 $ and using Cauchy-Schwarz inequality
\begin{align*}
\left| \int \d_y^2 v_\eps \d_y^2 \psi_\eps (\d_z+z \d_y) \chi^4 \right| & \leq \frac{1}{4 \eta_1 } \int (\d_y^2 \psi_\eps)^2 \chi^{4-2}+\eta_1\int |\d_y^2 v_\eps |^2 4^2 |(\d_z+z \d_y) \chi|^2 \chi^{4}                                        \\
& \leq  \frac{1}{4 \eta_1 } \int (\d_y^2 \psi_\eps)^2 \chi^{4-2}+\frac{\eta_1}{2}\int (|\d_y v_\eps|^2+|\d_y^3v_\eps|^2) 4^2 |(\d_z+z \d_y) \chi|^2 \chi^{4}             \\
& \leq \frac{1}{2} \int |\d_y^3 v_\eps|^2 \chi^4+\int \left(\frac{1}{2}|\d_y v_\eps|^2 +\frac{1}{4 \eta_1} |\d_y^2 \psi_\eps|^2\chi^2 \right) 
\end{align*}

and similarly with $\frac{1}{16} \eta_2^{-1}= \left(\|\theta''\|^2_\infty \frac{4}{y_0^2}+\frac{1}{L^2} \right)+\left(\|\theta'\|^2_\infty \frac{2}{y_0}+\frac{1}{L} \right)$
\begin{align*}
\left|\int \d_y^3 v_\eps \d_y v_\eps \d_y^2 \chi^4 \right|& \leq \frac{1}{4} \int |\d_y^3 v_\eps|^2 \chi^4+\int |\d_y v_\eps|^2 \left(4^2 |\d_y^2 \chi|^2 \chi^{4-2}+(4(4-1))^2 |\d_y \chi|^4 \chi^{4-4} \right)   \\
& \leq \frac{1}{4} \int |\d_y^3 v_\eps|^2 \chi^4+\frac{1}{\eta_2} \int |\d_y v_\eps|^2 
\end{align*}
\begin{align*}
\left| \int \d_y^4 \psi_\eps \d_y^2 \psi_\eps \d_y^2 \chi^4  \right| &\leq \frac{1}{4}  \int |\d_y^4 \psi_\eps|^2 \chi^4+\int |\d_y^2 \psi_\eps|^2 \left(4^2 |\d_y^2 \chi|^2 \chi^{4-2}+(4(4-1))^2 |\d_y \chi|^4 \chi^{4-4} \right)\\
&\leq \frac{1}{4}  \int |\d_y^4 \psi_\eps|^2 \chi^4+\frac{1}{\eta_2}\int |\d_y^2 \psi_\eps|^2. 
\end{align*}
Therefore combining all these previous inequalities we end up with

\begin{align*}
\int |\d_y^4 \psi_\eps|^2 \chi^4+\int |\d_y^3 v_\eps|^2 \chi^4 \leq c\int (\d_y^2 {s_\eps}_\psi \d_y^2 \psi_\eps \chi^4 +\d_y^2 {s_\eps}_v \d_y^2 v_\eps \chi^4) + C \left( \frac{1}{\eta_1}+\frac{1}{\eta_2}\right) \int (|\d_y v_\eps|^2+|\d_y^2 \psi_\eps|^2) 
\end{align*}

where $c,C$ are numerical constants.

Using the fact that $\chi \geq 0$, $\chi=1$ on $(y_0,L) \times (0,z_1)$ and taking $L \to + \infty$, we finally obtain

\begin{align*}
\int |\d_y^4 \psi_\eps|^2 \chi^4+\int |\d_y^3 v_\eps|^2 \chi^4 \leq C_{y_0,z_1} \left( \norm[E_0]{\bm{u}_\eps}^2+  \norm[E_0']{\d_y^2 {\bm{s}_\eps}}^2 \right).
\end{align*}

The claimed estimate follows from $\eps \to 0$.

This concludes the proof of theorem~\ref{existence}.
\section{Uniqueness and transparent boundary conditions}

In order to prove that \eqref{eq} admits a unique solution, we try to rely on an energy estimate. However the drawback of the weak formulation is that such an estimate makes no sense in the energy space as integrability with respect to $z$ is missing. In other words, we cannot take $\bm{u}$ as a test function. In this section we will show the uniqueness of the solution in a smaller space $\widetilde{E_0}$. It is to be noted that as the difficulties appear when $z \to \infty$, in case~\ref{itm:casB} we can recover uniqueness.

Once uniqueness is obtained we can reduce the study on the whole space to the study on a bounded (in $z$) domain thanks to so called transparent boundary conditions. We will exhibit such boundary conditions and in the last part briefly see their explicit formulation in a simple setting.

\subsection{Uniqueness}
The main obstacle to uniqueness is once more the lack of information with respect to $z$ in the energy space. More precisely if, instead of a degenerate elliptic equation, we consider \eqref{eq} as a transport equation, the transport term being $\d_z + z \d_y$ with a cross-diffusion term $\frac{1}{2}\begin{pmatrix}0 & \d_y^4& \\-\d_y^2 & 0\end{pmatrix}$, the main risk is the loss of mass along the characteristics $y-\frac{z^2}{2}=c$. Unfortunately, in the unbounded case we were not able to show that such a problem does not occur as such characteristics go to infinity. However up to a hypothesis of integrability we can show uniqueness of weak solutions.

Let $\widetilde{E_0}$ be the set of all functions $\bm{u}=(v,\psi) \in E_0$ such that $v \in L^2$ and $\psi \in L^2$ i.e

\begin{equation}
\label{defE0tilde}
\norm[\widetilde{E_0}]{\bm{u}}^2  = \int_\Omega \left( \left|\d_y v \right|^2 + \left|\frac{v}{1+y} \right|^2 +|v|^2\right) +\int_\Omega \left( \left|\d_y^2 \psi \right|^2 + \left|\frac{\psi}{1+y^2} \right|^2 + |\psi|^2\right)
\end{equation}

\begin{theorem}
\label{uniqueness}
There exists at most one solution of \eqref{eqweak} in $\widetilde{E_0}$.
\end{theorem}
\begin{proof}
As before we will focus on the case~\ref{itm:casA}, the other cases following similar analysis.               

By linearity it is sufficient to show that if $\bm{u} \in \widetilde{E_0}$ is a solution with homogeneous boundary conditions and $\bm{s}=0$ then $\bm{u}=0$. Let $\bm{u}$ be such a function.

The formal argument is the following. Define

\begin{equation*}
\mathcal{E}(Z)=\int_{0}^\infty v(y,Z) \psi(y,Z) dy.
\end{equation*}
We obtain differentiating with respect to $Z$
\begin{align*}
\frac{d \mathcal{E}}{dZ}= \frac{1}{2} \int_{0}^\infty \left|\d_y^2 \psi \right|^2(y,Z)+\left|\d_y v \right|^2(y,Z) dy \geq 0. 
\end{align*}

So $\mathcal{E}$ is $0$ at $Z=0$, $L^1$ and non-decreasing. The only option is then $\mathcal{E}=0$ almost everywhere. This leads to 
$$\int_\Omega |\d_y ^2 \psi|^2+|\d_y v|^2 =0$$
i.e $\psi=0$ and $v=0$ considering the boundary conditions.

However we cannot apply directly this formal argument as it requires to use $\bm{u}$ as a test function, which is not possible due to insufficient $z$ regularity, i.e $E_0 \not\subset E_1$.

So let $\bm{u}_\eps$ be the convolution with respect to $z$ of an approximation of the identity $\rho_\eps$ (with support in $\mathbb{R}^-$) with $\bm{u}$.

Then $\bm{u}_\eps \in E_1 \cap \widetilde{E_0}$, and the function

\begin{equation*}
\mathcal{E}_\eps (Z) = \int_0^\infty v_\eps(y,Z) \psi_\eps(y,Z) dy
\end{equation*}
is well defined in $L^1$. Moreover it is differentiable as $v_\eps,\psi_\eps \in C^\infty_z(L^2_y)$ and using the fact that
\begin{equation*}
\begin{aligned}
& \d_z v_\eps + z \d_y v_\eps - \frac{1}{2} \d_y^4 \psi_\eps= r^\psi_\eps \\
& \d_z \psi_\eps + z \d_y \psi_\eps + \frac{1}{2} \d_y^2 v_\eps= r^v_\eps 
\end{aligned}
\end{equation*}
where $\bm{r}_\eps=z \d_y \bm{u}_\eps-\rho_\eps*_z (z \d_y \bm{u})=(z \rho_\eps(z)) *_z \d_y \bm{u}$ (which goes to $0$ in $L^2$ when $\eps \to 0$), we obtain
\begin{equation*}
\label{Eepsz}
\frac{d \mathcal{E}_\eps}{dZ}(Z)= \frac{1}{2} \int_0^\infty \left(|\d_y^2 \psi_\eps|^2+|\d_y v_\eps|^2 \right) dy  + \int_0^\infty (r_\eps^\psi \psi_\eps+r_\eps^v v_\eps) dy.
\end{equation*}
So
\begin{equation*}
\mathcal{E}_\eps \to \mathcal{E} \text{ in } L^1.
\end{equation*}
and
\begin{equation*}
\frac{d \mathcal{E}_\eps}{d Z} \to \frac{1}{2} \int_{0}^\infty \left|\d_y^2 \psi \right|^2(y,Z)+\left|\d_y v \right|^2(y,Z) dy \text{ in } L^1.
\end{equation*}
From there $\frac{d \mathcal{E}}{ d Z} =\frac{1}{2} \int_{0}^\infty \left|\d_y^2 \psi \right|^2+\left|\d_y v \right|^2 dy$ as a distribution so
\begin{equation*}
\mathcal{E} \in W^{1,1}
\end{equation*}

To conclude it remains to show that the now well-defined $\mathcal{E}(0)$ is indeed $0$.

By the Caccioppoli inequality of the theorem~\ref{existence}, for all $a>0$ the trace $v_{|z=0,y>a}$ is well defined and $\int_a^\infty v(y,0) \psi(y,0) dy=0$, so
\begin{equation*}
\mathcal{E}(0)=0
\end{equation*}

The previously formal argument can now be used to obtain uniqueness.
\end{proof}

It should be noted that we can obtain the uniqueness in $E_0$ in the following variants:

\begin{itemize}
\item If the domain is bounded in $z$ (case~\ref{itm:casB}) then using the interior $z$ regularity, the boundaries at $z=0$ and $z=H$ and Poincaré inequalities in the $z$ variable we can recover a control of the $L^2$ norm of $u$ , the boundaries condition at $z=H$ leading to $\mathcal{E}(H) \leq 0$.
\item If the domain is $y \in \mathbb{R}, z>0$ then Fourier analysis leads easily to existence and uniqueness (see last subsection).
\item If the equation includes additional zero order terms then the natural energy space (dictated by $D$) is $\widetilde{E_0}$ instead of $E_0$ and thus include an $L^2$ control so the existence and uniqueness is assured (see next subsection). 
\item If there is no transport term then the equation is the same as the one for the $E^\frac{1}{3}$ Stewartson layer and uniqueness can once more be recovered with explicit Fourier analysis.
\end{itemize}

It is reasonable to hope that uniqueness indeed holds for the case~\ref{itm:casA} of \eqref{eq} but we need to have a better control along characteristics to show it.

\subsection{Transparent boundary conditions}

Similarly to the Dirichlet to Neumann operator for elliptic problems (used for example by Gerard-Varet and Masmoudi~\cite{gerard2010relevance} for Navier-Stokes equations), in this section we show that solving the equation on the whole space is equivalent to solving the same equation on the two subdomains ($0<z<H$ and $z>H$) with adequate boundary conditions on both subdomains.

Such a decomposition can be used to focus the study in a bounded (with respect to $z$) subdomain, which is especially useful for numerical analysis (as done in~\cite{marcotte2016equatorial}) and for deriving boundary layer operators as in~\cite{dalibard:hal-00258519}.

Unfortunately to make such a study a proper uniqueness result is needed. For this reason we will study variants of the initial problem, namely the one with additional zero order terms. It ensures that the energy norm controls the $L^2$ norm. The modified equation reads as

\begin{equation}
\label{eqzero}
\begin{aligned}
& \d_z v + z \d_y v - \frac{1}{2} \d_y^4 \psi -  \psi= s_\psi \\
& \d_z \psi + z \d_y \psi + \frac{1}{2} \d_y^2 v -v = s_v.     
\end{aligned}
\end{equation}

As before, the boundary conditions at $y=0$ will always be $v_{|y=0}=\d_y \psi_{|y=0}=0, \psi_{|y=0}=0$. The horizontal condition will be either~\ref{itm:casA},~\ref{itm:casB} or~\ref{itm:casC}.

The previous analysis leads to both existence and uniqueness for \eqref{eqzero}. With $\norm[\widetilde{E_1}]{\bm{\varpi}}=\norm[\widetilde{E_0}]{\bm{\varpi}}+\norm[\widetilde{E_0}']{\bm{\varpi}}$, where $\widetilde{E_0}$ is defined by \eqref{defE0tilde}, we have
\begin{lemma}
\label{theozero}
There exist a weak solution of \eqref{eqzero} in case~\ref{itm:casA},~\ref{itm:casB} and~\ref{itm:casC}.

This solution is unique and
\begin{equation*}\norm[\widetilde{E_0}]{\bm{u}} \leq \norm[\widetilde{E_1}']{\bm{u}}.\end{equation*}

Moreover in case~\ref{itm:casA}, if $\d_z \bm{s} \in \widetilde{E_1}'$ and $\bm{s} \in \widetilde{E_{0,0}}$ we have $\d_z \bm{u} \in \widetilde{E_0}$ and
\begin{equation*}\norm[\widetilde{E_0}]{\d_z \bm{u}} \leq C \left(\norm[\widetilde{E_1}']{\d_z \bm{s}}+\norm[\widetilde{E_0}]{\bm{s}}\right).\end{equation*}
\end{lemma}
\begin{proof}
The proof of this lemma is exactly the same as before, the only new point being the control on $\d_z \bm{u}$. This comes from the fact that in this particular case we can deduce boundary conditions on $\d_z \bm{u}$.

More precisely we have $\d_z \bm{u}$ verifying inside the domain
\begin{equation}
\label{eqdz}
\begin{aligned}
& \d_z (\d_z v) + z \d_y (\d_z v) - \frac{1}{2} \d_y^4 (\d_z \psi) -  \d_z \psi= \d_z s_\psi-\d_y v \\
& \d_z (\d_z \psi) + z \d_y (\d_z \psi) + \frac{1}{2} \d_y^2 (\d_z v) -\d_z v =  \d_z s_v    - \d_y  \psi 
\end{aligned}
\end{equation}
and the boundary conditions at $y=0$ are $\d_z v= \d_y \d_z \psi=0, \d_z \psi=0$. Moreover, contrary to the original problem \eqref{eq}, we have $(\d_y \psi, \d_y v) \in L^2$ and $\widetilde{E_1} \subset L^2$ so  
\begin{equation*}\|( \d_z s_v -\d_y \psi,\d_z s_\psi-\d_y v)\|_{\widetilde{E_1}'} \leq C \left(\norm[\widetilde{E_1}']{\d_z \bm{s}}+\norm[\widetilde{E_1}']{\bm{s}}\right). \end{equation*}
All that remains is the boundary condition at $z=0$. In case~\ref{itm:casA}, the equation \eqref{eq} leads to
\begin{equation*}\d_z v_{|z=0}=s_{\psi|z=0} \end{equation*}
which is an admissible boundary condition.
\end{proof}

Once we have obtained this result we can now consider transparent boundary conditions.
\begin{theorem}
\begin{enumerate}[label=(\roman*)]
\item \label{itm:vtopsii} ($v$-to-$\psi$ operator) For all $H>0$ there exists a non-positive operator $\Lambda_H: H^\frac{1}{2}_0 \to H^{-\frac{1}{2}}$ such that the only solution of \eqref{eqzero} in the domain $y>0,z>H$ with boundary condition $v_{|z=H}=V$ verifies $\psi_{|z=H}= \Lambda_H V $.
\item \label{itm:vtopsiii} (transparent BC) Let $H_0>0$ and let $\bm{s}$ verifying the hypothesis of lemma~\ref{theozero} be a source term with support inside $0<z<H_0$. For any $H>H_0$ let $\bm{u}^b$ be the solution of \eqref{eqzero} on the domain $y>0,H>z>0$ with boundary conditions of type~\ref{itm:casB}
\begin{equation*}
\psi_{|z=0}^b=0, \psi^b_{|z=H}=\Lambda_H v^b_{|z=H}
\end{equation*}
and let $\bm{u}_t$ be the solution of \eqref{eqzero} on $y>0,z>H$ with type~\ref{itm:casC} boundary condition
\[v^t_{z=H}=v^b_{|z=H}\]
and zero source term. Then $\bm{u}^b 1_{0<z<H}+\bm{u}^t 1_{z \geq H}$ is the solution of \eqref{eqzero} on the whole domain $y>0,z>0$ with boundary condition $\psi_{|z=0}=0$.
\end{enumerate}
\end{theorem}
\begin{proof}

We start by the point~\ref{itm:vtopsii}, i.e the definition of the operator $\Lambda_H$.

For $V \in H^\frac{1}{2}_0$ let $\bm{u}^V$the solution of \eqref{eqzero} in case~\ref{itm:casC} with nonhomogeneous boundary condition $v_{|z=H}=V$. Let us recall that such a solution is obtained by considering homogeneous boundary condition but with a source term $\bm{s}^V=L \bm{r}^V$ where $\bm{r}^V$ is an appropriate lifting.

Similarly for any $W \in H^\frac{1}{2}_0$ let $\bm{u}^W$ be the solution of \eqref{eqzero} with $w_{|z=H}=W$.

Using the same argument as in the proof of theorem~\ref{uniqueness}
\begin{equation*}
\mathcal{Q}(Z)=\int_0^\infty (v^W \psi^V ) (y,Z) dy
\end{equation*}
is well defined and in $W^{1,1}$ (note that $\mathcal{E}$ is the quadratic form associated with the bilinear form $\mathcal{Q}$) and

\begin{equation*}
\frac{ d \mathcal{Q}}{dZ} = \int_{0}^\infty \left( \d_y^2 \psi^V \d_y^2 \psi^W + \d_y v^V \d_y v^W +v^V v^W +\psi^V \psi^W \right)(y,Z) dy.
\end{equation*}

So as
\begin{align*}
\|\mathcal{Q} \|_{L^\infty} & \leq  C\|\mathcal{Q}\|_{W^{1,1}}  \leq C\left( \int |v^W \psi^V| +\int \left( \d_y^2 \psi^V \d_y^2 \psi^W + \d_y v^V \d_y v^W +v^V v^W+\psi^V \psi^W \right)\right) \\
& \leq C \|u^V\|_{\widetilde{E_0}} \|u^W\|_{\widetilde{E_0}} \leq C \|\bm{s}^V\|_{\widetilde{E_1}'} \|\bm{s}^W\|_{\widetilde{E_1}'}                                                                             \\
&  \leq C \|V\|_{H^\frac{1}{2}_0}\|W\|_{H^\frac{1}{2}_0}                                                                                 
\end{align*}
we obtain

\begin{equation*}
\forall W\in H^\frac{1}{2}_0, \left| \int_0^\infty W \psi^V_{|z=H} dy\right| \leq C  \|V\|_{H^\frac{1}{2}_0}\|W\|_{H^\frac{1}{2}_0}.
\end{equation*}

This means that $\psi^V_{|z=H} \in H^{-\frac{1}{2}}$ and moreover the application $\Lambda_H: V \mapsto \psi^V_{|z=H}$ is continuous from $H^\frac{1}{2}_0$ to its dual space.

At last since $\mathcal{Q}(Z) \to 0$ when $Z \to \infty$ we have
\begin{equation*}
\int_{0}^\infty V \Lambda_H V dy = \int_0^\infty V \psi^V_{|z=H} dy = -\int_H^\infty \int_0^\infty |\d_y^2 \psi^V|^2 +|\d_y v^V|^2+|v^V|^2+|\psi^V|^2 dy dz \leq 0.
\end{equation*}
and therefore $\Lambda_H$ is a non-positive operator.

It remains to prove~\ref{itm:vtopsiii} i.e that this condition is indeed a transparent boundary condition.

First of all let $\bm{u}$ be the solution of \eqref{eqzero} in case~\ref{itm:casA} and with source term $\bm{s}$.

Then by lemma~\ref{theozero} $v$ has proper trace in $H^\frac{1}{2}_0$ and $v_{|z=H}$, $\psi_{|z=H}$ are well defined. So $\psi_{|z=H} =\Lambda_H v_{|z=H}$ and $\bm{u}$ is a solution of \eqref{eqzero} so by uniqueness in the case~\ref{itm:casB} we have $\bm{u} 1_{0 \leq z \leq H} = \bm{u}^b$. 

We deduce that $\bm{u}^b_{|z=H}$ is well defined and is an admissible trace so $\bm{u}^t$ is well defined and once more by uniqueness in case~\ref{itm:casC} $\bm{u} 1_{z \geq} = \bm{u}^t$.

We can prove this result without solving the problem on the whole space: by constructing $\bm{u}$ from $\bm{u}^b$ and $\bm{u}^t$ in order to show that such $v$-to-$\psi$ operator is necessary to ensure the continuity of both $v$ and $\psi$.

With $\bm{u}=\bm{u}^b 1_{0 \leq z \leq H}+\bm{u}^t 1_{z \geq H}$ it is straightforward to see that the weak formulation on the whole space is verified for any test function with support inside $0<z<H$ or $z>H$.

Let $\bm{\varpi}=(w,\phi) \in \mathcal{D}$. Let $\chi$ be a smooth function such that $\chi(s)=0$ for $|s| > 2$ and $\chi(s)=1$ for $|s|<1$.

Then with $\chi_\eps(z)=\chi\left(\frac{z-H}{\eps} \right)$ using the fact that $\bm{\varpi}=\bm{\varpi} \chi_\eps +\bm{\varpi} (1-\chi_\eps)$ we obtain

\begin{equation*}
\langle L\bm{u}, \bm{\varpi} \rangle=\langle L\bm{u} , \bm{\varpi} (1-\chi_\eps) \rangle + \langle L \bm{u}, \bm{\varpi} \chi_\eps \rangle = \langle \bm{s} + 0 , \bm{\varpi} (1-\chi_\eps) \rangle +\langle L \bm{u}, \bm{\varpi} \chi_\eps \rangle 
\end{equation*}
as $\bm{\varpi} (1-\chi_\eps)$ is the sum of a function with support inside $0<z<H$ and a function with support inside $z>H$.

The last term is
\begin{align*}
\int_\Omega  - v  (\d_z\phi \chi_\eps+\phi \d_z \chi_\eps) - z v  \d_y\phi \chi_\eps-\frac{1}{2} \d_y^2 \psi \d_y^2 \phi \chi_\eps +\psi \phi \chi_\eps \\+\int_\Omega - \psi ( \d_z w \chi_\eps + w \d_z \chi_\eps) - z  \psi  \d_y w \chi_\eps -\frac{1}{2} \d_y v \d_y w \chi_\eps + v w \chi_\eps
\end{align*}
and we will show that it goes to $0$ when $\eps \to 0$.

Indeed as $\bm{u} \in \widetilde{E_0}$ and $\bm{\varpi} \in \widetilde{E_1}$ we have when $\eps \to 0$
\begin{equation*}
\int_\Omega  - v  \d_z\phi \chi_\eps - z v  \d_y\phi \chi_\eps-\frac{1}{2} \d_y^2 \psi \d_y^2 \phi \chi_\eps +\psi \phi \chi_\eps +\int_\Omega - \psi \d_z w \chi_\eps  - z  \psi  \d_y w \chi_\eps -\frac{1}{2} \d_y v \d_y w \chi_\eps + v w \chi_\eps \to 0.
\end{equation*}
As $\bm{\varpi}$ is identically $0$ near $y=0$ and $\bm{s}=0$ near $z=H$ using once more the same arguments as before
\begin{equation*}
\int_\Omega v \phi \d_z \chi_\eps + \psi w \d_z \chi_\eps \to \int_{y=0}^\infty \left((v^b-v^t) \phi + (\psi^b-\psi^t) w \right)(y,H) dy
\end{equation*}
which is zero, as the boundary conditions can be rewritten as $v^b-v^t=0$ and $\psi^b-\psi^t=\Lambda_H v^b-\psi^t=\Lambda_H v^b - \Lambda_H v^t=0$.
\end{proof}
\subsection{The case of the half plane}

In the case where the domain is the half-plane $z>0$ existence and uniqueness are a lot more easier. In fact we can use Fourier transform. Denoting by $\widehat{f}(\xi,z)$ the Fourier transform of $f(y,z)$ with respect to $y$ one can see that the problem can be rewritten as an ODE for each $\xi$ 

\begin{align*}
\begin{pmatrix} \d_z+z i \xi & -\xi^4 \\-\xi^2  & \d_z +z i \xi\end{pmatrix} \begin{pmatrix} \widehat{v} \\ \widehat{\psi} \end{pmatrix}=\begin{pmatrix} \widehat{s_\psi} \\ \widehat{s_v} \end{pmatrix}.
\end{align*}

Hence with $\widehat{w_\pm}=\widehat{v}  \pm |\xi| \widehat{\psi}$ the problem is diagonalized

\begin{align*}
\d_z \widehat{w}_\pm + (z i \xi  \mp |\xi|^3) \widehat{w_\pm} = \widehat{s}_\pm 
\end{align*}

and the explicit solution is
\begin{align*}
\widehat{w}_\pm(\xi,z)=\widehat{w}_\pm(0) e^{-\frac{z^2}{2}i \xi} e^{\pm |\xi|^3 z} + \int_0^z e^{-\frac{z^2-s^2}{2}i \xi}e^{\pm |\xi|^3 (z-s)} \widehat{s_\pm} (s) ds 
\end{align*}

Note that two exponential modes appear: one in $e^{-|\xi|^3 z}$ and one in $e^{|\xi|^3 z}$. To ensure that $\widehat{w}_+$ ansd $\widehat{w}_-$ are both in $L^2$, a necessary and sufficient condition is that the coefficient of $\exp(|\xi|^3z)$ is $0$. This offers another explanation of why only one condition at $z=0$ can be fixed.

For the transparent boundary condition, if there is no source term this condition simply becomes $\widehat{w}_+(\xi,H)=0$ i.e
\begin{equation*}
\forall \xi , \, \widehat{v} + |\xi| \widehat{\psi}=0
\end{equation*}
which in real space translates as $\Lambda_H = - (-\Delta)^\frac{-1}{2}$. This is exactly the condition used in~\cite{marcotte2016equatorial} for the numerical approximation.

It is to be noted that in this case the operator $\Lambda_H$ goes in fact from $\dot{H}^\frac{1}{2}$ to $\dot{H}^\frac{3}{2}$ which is the expected regularity as $\d_y \psi$ is of the same regularity as $v$.

But in our case because of the transport term we cannot use symmetries to extend \eqref{eq} to the whole half space.

\appendix
\section*{Appendix}
\subsection*{Physical derivation}
We recall here the main steps of the derivation of \eqref{eq} and refer to~\cite{marcotte2016equatorial} for details.

We consider the Stokes-Coriolis problem between two surfaces of revolution $\Gamma_\pm$ (our main focus will be spheres of radius $R_\pm$) and denote by $(X,\Phi,Z)$ the cylindrical coordinates. The Stokes equation of an incompressible fluid rotating around the axis $e_Z$ where we neglect the transport, in non-dimensional variables and with $E$ the Ekman number, can be written as
\begin{align*}
\grad p + e_Z \vect U - E \lapla U & = 0 \\
\div{U}                            & =0. 
\end{align*}
We consider non-penetration boundary conditions on $\Gamma_\pm$
\begin{equation*}
\begin{aligned}
&U\cdot n =0 \\
& U= V_\pm e_\Phi+\Upsilon_\pm e_\Phi \vect n.
\end{aligned}
\end{equation*}

If we consider an axisymmetric flow, $U=(U_X(X,Z),V(X,Z),U_Z(X,Z))$ then the incompressibility condition becomes $\d_X U_X+\d_Z U_Z=0$ so there exist a stream function $\Psi$ such that
\begin{equation*}
U= \begin{pmatrix} \d_Z \Psi \\ V \\- \d_X \Psi \end{pmatrix}
\end{equation*}
The corresponding equations are  
\begin{align*}
\label{geo}                         
\d_z V - E \lapla^2 \Psi=0 \\
\d_z \Psi + E \lapla V =0     
\end{align*}

and the boundary conditions
\begin{align*}
V_{|\Gamma^+}=V_+, V_{|\Gamma^-}=V_-                      \\
\d_n \Psi_{|\Gamma^+}=\Upsilon_+,\d_n \Psi_{|\Gamma^-}=\Upsilon_- \\
\Psi_{|\Gamma^\pm=0}=0                                      
\end{align*}

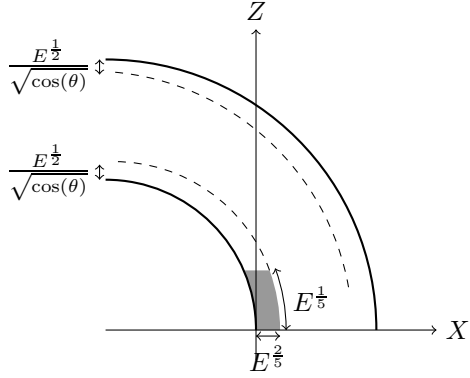
\begin{figure}
\centering
\label{fig:BLscaling}
\caption{ The different scalings and boundary layers}
\begin{tikzpicture}[scale=0.4]
    \fill[gray!80] (5.8,0) arc (0:20:5.8) -- (20:5.8)--({asin(5.8/5*sin(20))}:5) arc ({asin(5.8/5*sin(20))}:0:5)--(5.8,0);
	\fill[gray!80] (5.5,0) arc (0:2:5.5) -- (2:5.5)--(2:5) arc (2:0:5)--(5.5,0);
	\draw[dashed] (20:5.8) arc (20:90:5.5);
	\draw[dashed] (10:8.2) arc (10:85:8.7);
    \draw[<->] (6,0) arc (0:20:6) node[midway,right] {$E^\frac{1}{5}$};
	\draw[<->] (5,-0.2) --(5.8,-0.2)node[midway,below] {$E^\frac{2}{5}$} ;
	\draw[<->] (-0.2,5) -- (-0.2,5.5) node[midway,left] {$\frac{E^\frac{1}{2}}{\sqrt{\cos(\theta)}}$};
	\draw[<->] (-0.2,8.5) -- (-0.2,9) node[midway,left] {$\frac{E^\frac{1}{2}}{\sqrt{\cos(\theta)}}$};
	\draw[thick] (5,0) arc (0:90:5);
    \draw[thick] (9,0) arc (0:90:9);
    \draw[->,thin] (0,0) --(11,0) node[right] {$X$};
    \draw[->,thin] (5,-1) --(5,10) node[above] {$Z$};
\end{tikzpicture}
\end{figure}
When $E \to 0$ we obtain the formal equations $\d_Z V=0$, $\d_Z \Psi=0$. So at the main order in $E$, inside the domain
\begin{equation*} 
\begin{aligned}
&V(X,Z)=V^0(X)+o(1)\\
&\Psi(X,Z)=0+o(1)
\end{aligned}
\end{equation*}
In order to find $V^0$ and pursue further the asymptotic expansion we must consider the boundary layer ensuring that the boundary conditions are satisfied. 

Near a horizontal boundary (i.e constant $Z$) we recover the classical Ekman scaling
\begin{equation*}
\begin{aligned}
&V=v \left(X,\frac{Z}{E^\frac{1}{2}}\right)\\
&\Psi=E^\frac{1}{2} \psi \left(X,\frac{Z}{E^\frac{1}{2}}\right)
\end{aligned}
\end{equation*}
and with $(x,z)$ the rescaled variables the boundary equation is
\begin{align*}
\d_z v - \d_z^4 \psi=0 \\
\d_z  \psi +  \d_z^2 v =0.
\end{align*}

Note that the same equation holds for any boundary as long as  $\cos(\theta)= e_Z \cdot n$ does not approach $0$ where $\theta$ is the angle between the normal of the surface and the axis of rotation. In this case the scaling is 
\begin{equation*}
Z=\frac{z}{E^\frac{1}{2} \cos(\theta)^{-\frac{1}{2}}}
\end{equation*}

For a vertical boundary (i.e constant $X$) the scaling is
\begin{equation*}
\begin{aligned}
&V=v \left(\frac{X}{E^\frac{1}{3}},Z\right)\\
&\Psi=E^\frac{1}{3} \psi \left(\frac{X}{E^\frac{1}{3}},Z\right)
\end{aligned}
\end{equation*}
and the equation
\begin{align*}
\d_z v - \d_y^4\psi=0 \\
\d_z \psi + \d_y^2 v =0. 
\end{align*}

In the case of $\cos(\theta)$ approaching $0$ the previous scaling and equation are no longer correct.

If the boundary is $Z = (-X)^\alpha 1_{X<0}$, denoting by $Y=X+Z^\frac{1}{\alpha}$ the equation becomes
\begin{equation*}
\begin{aligned}
&\left(\d_Z+\alpha^{-1} Z^\frac{1-\alpha}{\alpha}\d_Y\right) V - E \left(\d_Z^2+\d_Y^2+2\alpha^{-1} Z^\frac{1-\alpha}{\alpha} \d_Y \d_Z\right)^2 \Psi  =0\\
&\left(\d_Z+\alpha^{-1} Z^\frac{1-\alpha}{\alpha}\d_Y\right)\Psi + E \left(\d_Z^2+\d_Y^2+2\alpha^{-1} Z^\frac{1-\alpha}{\alpha} \d_Y \d_Z\right) V = 0
\end{aligned}
\end{equation*}
The  scaling is then
\begin{align*}
V    & =v \left( \frac{y}{E^\frac{1}{3-\alpha}},\frac{z}{E^\frac{\alpha}{3-\alpha}}\right)                         \\
\Psi & =E^\frac{1}{3-\alpha} \psi \left(\frac{y}{E^\frac{1}{3-\alpha}},\frac{z}{E^\frac{\alpha}{3-\alpha}} \right) 
\end{align*}

and the associated equation becomes
\begin{align*}
& \d_z v+\alpha^{-1}z^\frac{1-\alpha}{\alpha} \d_y v - \d_y^4 \psi-E^\frac{4(1-\alpha)}{(3-\alpha)} \d_z^4 \psi =0 \\
& \d_z \psi +\alpha^{-1}z^\frac{1-\alpha}{\alpha} \d_y \psi   + \d_y^2 v +E^\frac{2(1-\alpha)}{(3-\alpha)}  \d_z^2 v =0                                                   
\end{align*}
with domain $y>0,z>0$.

The higher terms in $\d_z$ lead to another boundary layer of size $E^\frac{3(1-\alpha)}{2(3-\alpha)}$ in $z$ i.e a standard Ekman layer of size $E^{\frac{3(1-\alpha)}{2(3-\alpha)}+\frac{\alpha}{1-\alpha}}=E^\frac{1}{2}$. Note that this Ekman layer can be expressed in term of a boundary condition connecting $v$ and $\d_z \psi$ but that in the physical case it is simply a symmetry condition, $\psi=0$. 

Considering only the higher order in $E$ we obtain the announced equation for the spherical case $\alpha=\frac{1}{2}$.

Note that there are other boundary layers in the vicinity of the equator or of the cylinder $X=R_-$, but since we do not describe them in this paper we did not include them in figure~\ref{fig:BLscaling}. We refer to \cite{stewartson1957almost,marcotte2016equatorial} for a complete physical description.

\subsection*{Duality argument}

To prove existence of a solution we used a simpler version of Lions-Lax-Milgram~\cite{lions1966remarks} which can be rewritten as:
\begin{lemma}
\label{analfunc}
Let $E$ and $F$ two reflexive Banach spaces and
\begin{equation*}L:E \rightarrow F'\end{equation*}
a continuous operator.

Let $D \subset F$ a dense subspace of $F$ and $L^*$ the adjoint of $L$ from $F$ to $E'$.

If there exists a constant $\gamma >0$ such that
\begin{equation}
\label{coercivity}
\forall v \in D, \norm[E']{L^* v} \geq \gamma \norm[F]{v}
\end{equation}
then for all $f \in F'$ there exist a solution $u$ of
\begin{equation*}L u =f\end{equation*}
with
\begin{equation*}
\norm[E]{u} \leq \frac{1}{\gamma} \norm[F']{f}
\end{equation*}
\end{lemma}
The proof is elementary but as we did not find this exact formulation in the literature we detail the proof for the reader's convenience. 

Let us first notice that the relation \eqref{coercivity}, also called observability inequality, ensures that $L^*$ is injective. Thus the linear form
\begin{align*}
\phi: L^* D & \rightarrow \R             \\
L^* v                   & \mapsto \dual{f}{v}_{F',F} 
\end{align*}
is well defined. Moreover it is continuous
\begin{equation*}
|\phi(L^*v)|=|\dual{f}{v}_{F',F}| \leq \norm[F']{f} \norm[F]{v} \leq \frac{1}{\gamma} \norm[F']{f} \norm[E']{L^* v}.
\end{equation*}
As $D$ is dense we can define $\phi$ as a continuous form on $L^* F \subset E'$.

Using Hahn-Banach theorem, we then extend $\phi$ as a linear continuous form on the whole $E'$.
As $E$ is a reflexive Banach space there exists $u \in E$ such that
\begin{align*}
\forall g \in E',\dual{u}{g}_{E,E'}=\phi(g) 
\end{align*}
and in particular
\begin{align*}
\forall L^*v \in L^* F, \dual{u}{L^* v}_{E,E'}=\phi(L^* v)=\dual{f}{v}_{F',F} 
\end{align*}
i.e
\begin{align*}
\forall v \in F, \dual{L u}{v}_{F',F}=\dual{f}{v}_{F',F}.
\end{align*}

\section*{Acknowledgements}
This project has received funding from the European Research Council (ERC) under the European Union's Horizon 2020 research and innovation program Grant agreement No 637653, project BLOC ``Mathematical Study of Boundary Layers in Oceanic Motion''. 	
\bibliographystyle{amsplain}
\bibliography{refekman}

\providecommand{\bysame}{\leavevmode\hbox to3em{\hrulefill}\thinspace}
\providecommand{\MR}{\relax\ifhmode\unskip\space\fi MR }
\providecommand{\MRhref}[2]{%
  \href{http://www.ams.org/mathscinet-getitem?mr=#1}{#2}
}
\providecommand{\href}[2]{#2}
\begin{thebibliography}{10}

\bibitem{bresch2004rotating}
Didier Bresch, Beno{\i}t Desjardins, and David G{\'e}rard-Varet, \emph{Rotating
  fluids in a cylinder}, Discrete \& Continuous Dynamical Systems-A \textbf{11}
  (2004), no.~1, 47--82.

\bibitem{chemin2006mathematical}
Jean-Yves Chemin, Benoit Desjardins, Isabelle Gallagher, and Emmanuel Grenier,
  \emph{Mathematical geophysics: An introduction to rotating fluids and the
  navier-stokes equations}, vol.~32, Oxford University Press on Demand, 2006.

\bibitem{dalibard:hal-00258519}
Anne-Laure Dalibard and Laure Saint-Raymond, \emph{{Mathematical study of
  resonant wind-driven oceanic motions}}, {Journal of Differential Equations}
  \textbf{246} (2009), no.~6, 2304--2354.

\bibitem{dalibard2010mathematical}
\bysame, \emph{{Mathematical study of the $\beta$-plane model for rotating
  fluids in a thin layer}}, Journal de math{\'e}matiques pures et
  appliqu{\'e}es \textbf{94} (2010), no.~2, 131--169.

\bibitem{fichera1959unified}
Gaetano Fichera, \emph{On a unified theory of boundary value problems for
  elliptic-parabolic equations of second order}, Mathematics Research Center,
  United States Army, University of Wisconsin, 1959.

\bibitem{gerard2010relevance}
David G{\'e}rard-Varet and Nader Masmoudi, \emph{Relevance of the slip
  condition for fluid flows near an irregular boundary}, Communications in
  Mathematical Physics \textbf{295} (2010), no.~1, 99--137.

\bibitem{gerard2008remarks}
David Gerard-Varet and Thierry Paul, \emph{{Remarks on boundary layer
  expansions}}, Communications in Partial Differential Equations \textbf{33}
  (2008), no.~1, 97--130.

\bibitem{grenier1997ekman}
Emmanuel Grenier and Nader Masmoudi, \emph{Ekman layers of rotating fluids, the
  case of well prepared initial data}, Communications in Partial Differential
  Equations \textbf{22} (1997), no.~5-6, 213--218.

\bibitem{grisvard2011elliptic}
Pierre Grisvard, \emph{Elliptic problems in nonsmooth domains}, vol.~69, SIAM,
  2011.

\bibitem{gerard2005formal}
David Gérard-Varet, \emph{{Formal derivation of boundary layers in fluid
  mechanics}}, Journal of Mathematical Fluid Mechanics \textbf{7} (2005),
  no.~2, 179--200.

\bibitem{lions1966remarks}
Jacques~Louis Lions, \emph{Remarks on evolution inequalities}, Journal of the
  Mathematical Society of Japan \textbf{18} (1966), no.~4, 331--342.

\bibitem{marcotte2016equatorial}
Florence Marcotte, Emmanuel Dormy, and Andrew Soward, \emph{On the equatorial
  ekman layer}, Journal of Fluid Mechanics \textbf{803} (2016), 395--435.

\bibitem{masmoudi2011hardy}
Nader Masmoudi, \emph{About the hardy inequality}, An Invitation to
  Mathematics, Springer, 2011, pp.~165--180.

\bibitem{proudman1956almost}
Ian Proudman, \emph{{The almost-rigid rotation of viscous fluid between
  concentric spheres}}, Journal of Fluid Mechanics \textbf{1} (1956), no.~5,
  505--516.

\bibitem{rousset2007asymptotic}
Fr{\'e}d{\'e}ric Rousset, \emph{{Asymptotic behavior of geophysical fluids in
  highly rotating balls}}, Zeitschrift f{\"u}r angewandte Mathematik und Physik
  \textbf{58} (2007), no.~1, 53--67.

\bibitem{stewartson1957almost}
K~Stewartson, \emph{{On almost rigid rotations}}, Journal of Fluid Mechanics
  \textbf{3} (1957), no.~1, 17--26.

\bibitem{stewartson1966almost}
\bysame, \emph{{On almost rigid rotations. Part 2}}, Journal of fluid mechanics
  \textbf{26} (1966), no.~1, 131--144.

\bibitem{van1992stewartson}
AI~Van~de Vooren, \emph{The stewartson layer of a rotating disk of finite
  radius}, Journal of engineering mathematics \textbf{26} (1992), no.~1,
  131--152.

\end{thebibliography}
\end{document}